\documentclass[preprint,11pt]{elsarticle}    

\usepackage{graphicx,epsfig}
\usepackage{enumerate}
\usepackage{amssymb}
\usepackage{amsthm}
\usepackage{amsmath}
\usepackage{centernot}
\usepackage{xcolor}
\usepackage{todonotes}
\usepackage{times}
\usepackage{mathrsfs}
\usepackage[utf8]{inputenc}
\usepackage{subfigure}
\usepackage{multirow}
\usepackage{setspace}

\usepackage{mathtools}
\usepackage{appendix}

\usepackage{amsmath}
\usepackage{wrapfig}
\usepackage{setspace}
\usepackage{booktabs}
\usepackage{lscape}
\usepackage{algorithm,algorithmicx}
\usepackage{algcompatible}
\usepackage{bm}
\usepackage{graphicx,epsfig}
\usepackage{xfrac}
\usepackage{url}
\usepackage{multirow}
\usepackage{longtable}
\usepackage[linkcolor=blue, urlcolor=blue, citecolor=blue, colorlinks]{hyperref}

\newtheorem{thm}{Theorem}[section]
\theoremstyle{definition}
\newtheorem{dfn}{Definition}[section]
\newtheorem{lem}{Lemma}[section]
\newtheorem{example}{Example}[section]

\newtheorem{note}{Note}[]

\makeatletter
\newcommand{\thickhline}{%
	\noalign {\ifnum 0=`}\fi \hrule height 1pt
	\futurelet \reserved@a \@xhline
}

\journal{Journal of Computational and Applied Mathematics}
\begin{document}
\begin{frontmatter}
\title{\textbf{Generalized-Hukuhara Subgradient Method for Optimization Problem with Interval-valued Functions and its Application in Lasso Problem}}
\author[iitbhu_math]{Debdas Ghosh\corref{cor1}}
        \ead{debdas.mat@iitbhu.ac.in}
\author[iitbhu_math]{Amit Kumar Debnath}
        \ead{amitkdebnath.rs.mat18@itbhu.ac.in}
\author[radko,radko2]{Radko Mesiar}
        \ead{mesiar@math.sk}
\author[iitbhu_math]{Ram Surat Chauhan}
        \ead{rschauhan.rs.mat16@itbhu.ac.in}
\address[iitbhu_math]{Department of Mathematical Sciences,  Indian Institute of Technology (BHU) Varanasi \\ Uttar Pradesh--221005, India}
\cortext[cor1]{Corresponding author}
\address[radko]{Faculty of Civil Engineering, Slovak University of Technology, Radlinsk\'{e}ho 11, 810 05, Bratislava}
\address[radko2]{Palack\'y University Olomouc, Faculty of Science, Department of Algebra and Geometry, 17. listopadu 12, 771 46 Olomouc, Czech Republic}		 

\begin{abstract}
  In this study, a \emph{$gH$-subgradient technique} is developed to obtain efficient solutions to the optimization problems with nonsmooth nonlinear convex interval-valued functions. The algorithmic implementation of the developed $gH$-subgradient technique is illustrated. As an application of the proposed \emph{$gH$-subgradient technique}, an $\ell_1$ penalized linear regression problem, known as a  \emph{lasso problem}, with interval-valued features is solved.
\end{abstract}
\begin{keyword}
Interval-valued functions\sep  Interval optimization problems\sep Convexity\sep  $gH$-subgradient\sep $gH$-subgradient method\sep Lasso problems.
\\ \vspace{0.7cm}
\end{keyword}

\end{frontmatter}

\section{\textbf{Introduction}}

\noindent The subgradient methods \cite{Bertsekas1999,Boyd2004} are useful techniques to solve nonsmooth convex optimization problems. The articles that are referred in \cite{Beck2003,Kiwiel1983,Nedic2001,Nesterov2009,Studniarski1989} provide various subgradient techniques for conventional convex optimization problems. Since optimization problems with Interval-Valued Functions (IVFs), known as Interval Optimization Problems (IOPs), become substantial topics to the researchers due to inexact and imprecise natures of many real-world occurrences, in this article, we develop a subgradient technique to solve the convex IOPs.

\subsection{Literature survey}

\noindent In order to deal with compact intervals and IVFs, Moore \cite{Moore1966}  proposed the arithmetic of intervals. It is observed that Moore's interval arithmetic cannot provide the additive inverse of an interval whose upper and lower limits are different, i.e., a nondegenerate interval. For this reason, Hukuhara \cite{Hukuhara1967} introduced a new concept for the difference of intervals, known as $H$-difference of intervals which provides the additive inverse of any compact interval. However, the $H$-difference is not applicable between all pairs of compact intervals (see \cite{Ghosh2019derivative} for details). To overcome this difficulty, in the literature of IVFs and IOPs, the `nonstandard subtraction', introduced by Markov \cite{Markov1979}, has been used and named as $gH$-difference by Stefanini \cite{Stefanini2008}. The $gH$-difference not only provides the additive inverse of any compact interval but also it is applicable for all pairs of compact intervals.\\

In the field of optimization, to find out the best element or one of the best elements, an ordering is inevitably required. Since unlike the real numbers, intervals are not linearly ordered, most of the researches \cite{Chalco2013kkt,Chalco2013-2,Ghosh2019derivative,Ghosh2019gradient,Stefanini2008,Stefanini2009,Wu2007,Wu2008} developed the theories of IVFs based on partial ordering relations of intervals proposed by Isibuchi and Tanka \cite{Ishibuchi1990}. In  \cite{Bhurjee2012} some ordering relations based on the parametric representation of intervals are proposed. However, all the ordering relations of \cite{Bhurjee2012} can be derived from the ordering relations of \cite{Ishibuchi1990}. The concept of variable ordering relation of intervals is introduced in \cite{Ghosh2020ordering}.\\

In order to develop the calculus of IVFs, the concept of differentiability of IVFs was initially introduced by Hukuhara  \cite{Hukuhara1967} with the help of $H$-difference of intervals. However, this definition of Hukuhara differentiability is restrictive \cite{Chalco2013-2}. Consequently, based on $gH$-difference, the concepts of $gH$-derivative, $gH$-partial derivative, $gH$-gradient, and $gH$-differentiability for IVFs are provided in \cite{Chalco2013kkt,Ghosh2016newton,Markov1979,Stefanini2009,Stefanini2019}. In \cite{Lupulescu2014}, Lupulescu developed the fractional calculus for IVFs. Recently, Ghosh et al.\ \cite{Ghosh2019derivative} have introduced the idea of $gH$-directional derivative, $gH$-G\^{a}teaux derivative, and $gH$-Fr\'echet derivative of IVFs. Further, in \cite{Ghosh2019gradient}, a new concept of $gH$-differentiability that is equipped with a linearity concept of IVFs is reported.\\

In the literature of IOP, there are various techniques and theorems for obtaining the solutions to IOPs (see \cite{Bhurjee2012,Chalco2013kkt,Chen2004,Ghosh2016newton,Ghosh2017spc,Ghosh2017quasinewton,Ishibuchi1990,Jiang2015,Liu2007,Lupulescu2014,Wu2008}). Ishibuchi and Tanaka \cite{Ishibuchi1990} proposed a method to solve linear IOPs, which is subsequently generalized in \cite{Chanas1996}. A numerical technique to solve quadratic IOPs has been proposed by Liu and Wang \cite{Liu2007}. For general nonlinear IOPs, Wu \cite{Wu2008} presented some conditions to obtain the solutions to an IOP.  Ghosh \cite{Ghosh2016newton} developed a Newton method and a quasi-Newton method \cite{Ghosh2017quasinewton} to solve IOPs. Chalco-Cano et al. \cite{Chalco2013kkt}  presented Karush-Kuhn-Tucker (KKT) conditions for IOPs. Chen \cite{Chen2020} proposed KKT optimality conditions for IOPs on Hadamard manifolds. All the aforementioned theories and techniques to obtain the solutions to IOPs are developed by converting the IOPs either into real-valued multiobjective optimization problems or into a real-valued single objective optimization problems using parametric representations of their interval-valued objective and constraint functions. However, the multiobjective approach is practically difficult (please see \cite{Ghosh2020ordering} for details). Further, it is to mention that for the parametric representation of an IVF one needs its explicit form which is also practically difficult, for instance, consider the function $\boldsymbol{\mathcal{E}} (\beta)$ in (\ref{eef}) of the present article.



\subsection{Motivation and contribution of the paper}
\noindent Since using parametric representations of the corresponding IVFs of IOPs or converting the IOPs into real-valued multiobjective optimization problems are practically quite difficult, recently, researchers (for instance, see \cite{Ghosh2019extended,Ghosh2019gradient,Stefanini2019}, etc.) are trying to develop the theories and techniques to obtain the solutions to IOPs without using a real-valued multiobjective approach or parametric approach. Stefanini and Arana-Jim{\'e}nez \cite{Stefanini2019} proposed KKT conditions and Ghosh et al. \cite{Ghosh2019extended} proposed generalized KKT conditions to obtain the solution of the IOPs. In \cite{Ghosh2019gradient}, Ghosh et al. developed a few gradient-descent techniques for smooth IOPs. However, a technique to deal with nonsmooth IOPs, without using real-valued multiobjective approach or parametric approach, is still undeveloped.\\

In this article, alike to the conventional subgradient technique \cite{Bertsekas1999,Boyd2004} for real-valued optimization problems, we illustrate a $gH$-subgradient technique to obtain efficient solutions of nonsmooth convex IOPs. The algorithmic implementation and convergence of the proposed method are illustrated. As an application of the proposed $gH$-subgradient technique, a lasso problem with interval-valued features is solved.

\subsection{Delineation}
\noindent Rest of the article is delineated as follows. Some basic terminologies and notions of intervals along with convexity and a few topics of the calculus of IVFs are depicted in the next section. Section \ref{ssm} is devoted to a $gH$-subgradient technique to obtain efficient solutions to a convex IOP with its algorithmic implementation. Applying the proposed $gH$-subgradient technique, a lasso problem with interval-valued features is solved in Section \ref{sa}. Finally, the last section is concerned with a few future directions for our study.

\section{\textbf{Preliminaries and Terminologies}}
	
\noindent In this section, we discuss Moore's interval arithmetic \cite{Moore1966,Moore1987} followed by the concepts of $gH$-difference and special multiplication of two intervals, the dominance relations of intervals, and norm of intervals. Thereafter, we shall have a glance at convexity and a few topics of the calculus of IVFs. The ideas and notations that we describe in this section are used throughout the paper.

\subsection{Arithmetic of intervals and their dominance relation}\label{ssai}
	
\noindent Let us denote $\mathbb{R}$,  $\mathbb{R}_+$, and $I(\mathbb{R})$ as the set of real numbers, the set of all nonnegative real numbers, and the set of all compact intervals, respectively. We represent the elements of $I(\mathbb{R})$ by bold capital letters ${\textbf A}, {\textbf B}, {\textbf C}, \ldots $. Also, we represent an element $\textbf{A}$ of $I(\mathbb{R})$ with the help of the corresponding small letter in the following way
\[
\textbf{A} = [\underline{a}, \overline{a}],~\text{where}~\underline{a}\leq \overline{a}.
\]
It is notable that any singleton set $\{p\}$ of $\mathbb{R}$ can be represented by the interval $[\underline{p},\;\overline{p}]$ with  $\underline{p}=p=\overline{p}$. In particular,
\[
\textbf{0}=\{0\}=[0, 0]~~\text{and}~~\textbf{1}=\{1\}=[1, 1].
\]

The Moore's interval addition ($\oplus$), substraction ($\ominus$), multiplication ($\odot$), and division ($\oslash$) \cite{Moore1966,Moore1987} are defined as follows.
\begin{align*}
&\textbf{A} \oplus \textbf{B} = \left[\underline{a} + \underline{b}, \overline{a} +
\overline{b}\right],~\textbf{A} \ominus \textbf{B} = \left[\underline{a} - \overline{b}, \overline{a} -
\underline{b}~\right],\\
&\textbf{A} \odot \textbf{B}  = \left[\min\left\{\underline{a}\underline{b}, \underline{a}\overline{b},\overline{a}\underline{b}, \overline{a}\overline{b}\right\}, \max\left\{\underline{a}\underline{b}, \underline{a}\overline{b},\overline{a}\underline{b}, \overline{a}\overline{b}\right\} \right],\\
&\textbf{A} \oslash \textbf{B} = \left[\min\left\{\underline{a}/\underline{b}, \underline{a}/\overline{b}, \overline{a}/\underline{b}, \overline{a}/\overline{b}\right\}, \max\left\{\underline{a}/\underline{b}, \underline{a}/\overline{b}, \overline{a}/\underline{b}, \overline{a}/\overline{b}\right\} \right], \text{ provided } 0\not\in \textbf{B}.
\end{align*}
Since $\textbf{A}\ominus\textbf{A}\neq\textbf{0}$ for a nondegenerate interval $\textbf{A}$, in this article, we use the following concept of difference for intervals.
\begin{dfn}($gH$-{\it difference of intervals} \cite{Markov1979,Stefanini2008}). Let $\textbf{A}$ and $\textbf{B}$ be two elements of $I(\mathbb{R})$. The $gH$-difference
between $\textbf{A}$ and $\textbf{B}$, denoted by $\textbf{A} \ominus_{gH} \textbf{B}$, is defined by the interval $\textbf{C}$ such that
\[
\textbf{A} =  \textbf{B} \oplus  \textbf{C} ~\text{ or }~ \textbf{B} = \textbf{A}
\ominus \textbf{C}.
\]
It is to be noted that for $\textbf{A} = \left[\underline{a},~\overline{a}\right]$ and $\textbf{B} = \left[\underline{b},~\overline{b}\right]$,
\[
\textbf{A} \ominus_{gH} \textbf{B} = \left[\min\{\underline{a}-\underline{b},
\overline{a} - \overline{b}\},~ \max\{\underline{a}-\underline{b}, \overline{a} -
\overline{b}\}\right],
\]
and
\[
\textbf{A} \ominus_{gH} \textbf{A} = \textbf{0}.
\]
\end{dfn}
Furthermore, from the definition of multiplication of intervals, we observe that $\textbf{A}\odot\textbf{A} \not\subset I(\mathbb{R}_+)$, when $\underline{a}$ and $\overline{a}$ are of opposite signs. That is why we use the following special multiplication of intervals.
\begin{dfn}
(\emph{Special multiplication of intervals}). Let
$\textbf{A}$ and $\textbf{B}$ be two elements of $I(\mathbb{R})$. The special multiplication
of $\textbf{A}$ and $\textbf{B}$ is denoted $\textbf{A} \odot_{S} \textbf{B}$ and defined by
\[
 {\textbf A} \odot_{S} {\textbf B} = \left[\min\left\{\underline{a}\, \underline{b},\,\overline{a}\, \overline{b} \right\},\, \max\left\{\underline{a}\, \underline{b},\,\overline{a}\, \overline{b} \right\}\right].
\]
This special multiplication $\odot_S$ will be used in Section \ref{sa} for the convexity of IVFs.
\end{dfn}
The algebraic operations on the product space $I(\mathbb{R})^n=I(\mathbb{R})\times I(\mathbb{R})\times \cdots \times I(\mathbb{R})$ ($n$ times) are defined as follows.
\begin{dfn}(\emph{Algebraic operations on $I(\mathbb{R})^n$} \cite{Ghosh2019gradient})\label{daoirn}. Let $\widehat{\textbf{A}} = \left(\textbf{A}_1, \textbf{A}_2, \ldots, \textbf{A}_n\right)$ and $\widehat{\textbf{B}} = \left(\textbf{B}_1, \textbf{B}_2, \ldots, \textbf{B}_n\right)$ be two elements of $I(\mathbb{R})^n$. An algebraic operation $\star$ between $\widehat{\textbf{A}}$ and $\widehat{\textbf{B}}$, denoted by $\widehat{\textbf{A}} \star \widehat{\textbf{B}}$, is defined by
\[
\widehat{\textbf{A}} \star \widehat{\textbf{B}}=\left(\textbf{A}_1\star \textbf{B}_1, \textbf{A}_2\star \textbf{B}_2, \ldots, \textbf{A}_n\star \textbf{B}_n\right),
\]
where $\star\in\{\oplus,\ \ominus, \ \ominus_{gH}\}$.
\end{dfn}
\begin{dfn}({\it Dominance relations on intervals} \cite{Ghosh2019gradient,Wu2007}). Let $\textbf{A}$ and $\textbf{B}$ be two intervals in $I(\mathbb{R})$.
\begin{enumerate}[(i)]
\item $\textbf{B}$ is said to be dominated by $\textbf{A}$ if $\underline{a}\leq \underline{b}$ and $\overline{a}\leq\overline{b}$, and then we write $\textbf{A}\preceq \textbf{B}$;
\item $\textbf{B}$ is said to be strictly dominated by $\textbf{A}$ if either $\underline{a} \leq \underline{b}$  and $\overline{a} < \overline{b}$ or $\underline{a} < \underline{b}$  and $\overline{a} \leq \overline{b}$, and then we write $\textbf{A}\prec \textbf{B}$;
\item if $\textbf{B}$ is not dominated by $\textbf{A}$, then we write $\textbf{A}\npreceq \textbf{B}$; if $\textbf{B}$ is not strictly dominated by $\textbf{A}$, then we write $\textbf{A}\nprec \textbf{B}$;
\item if $\textbf{A}\npreceq \textbf{B}$ and $\textbf{B}\npreceq \textbf{A}$, then we say that none of $\textbf{A}$ and $\textbf{B}$ dominates the other, or $\textbf{A}$ and $\textbf{B}$ are not comparable.
\end{enumerate}
\end{dfn}
\begin{dfn}\label{irnorm}
(\emph{Norm on $I(\mathbb{R})$} \cite{Moore1966}). For an $\textbf{A} = \left[\underline{a}, \bar{a}\right]$ in $ I(\mathbb{R})$, the function ${\lVert \cdot \rVert}_{I(\mathbb{R})} : I(\mathbb{R}) \rightarrow \mathbb{R}_+$, defined by
\[
{\lVert \textbf{A} \rVert}_{I(\mathbb{R})} = \max \{|\underline{a}|, |\bar{a}|\},
\]
is a norm on $I(\mathbb{R})$.
\end{dfn}

In the rest of the article, we use the symbol `${\lVert \cdot \rVert}_{I(\mathbb{R})}$' to denote the norm on $I(\mathbb{R})$ and we simply use the symbol `${\lVert \cdot \rVert}$' to denote the usual Euclidean norm on $\mathbb{R}^n$.

\subsection{Convexity and calculus of IVFs}

\noindent A function $\textbf{F}$ from a nonempty subset $\mathcal{X}$ of $\mathbb{R}^n$ to $I(\mathbb{R})$ is known as an IVF (interval-valued function). At each $x \in \mathcal{X}$, the value of $\textbf{F}$ is presented by
\[
\textbf{F}(x)=\left[\underline{f}(x),\
\overline{f}(x)\right],
\]
where $\underline{f}(x)$ and $\overline{f}(x)$ are real-valued function on $\mathcal{X}$ such that $\underline{f}(x)\leq \overline{f}(x)$ for all $x \in \mathcal{X}$.
\begin{dfn}(\emph{Convex IVF} \cite{Wu2007}). Let $\mathcal{X} \subseteq \mathbb{R}^n$ be a convex set. An IVF $\textbf{F}: \mathcal{X} \rightarrow I(\mathbb{R})$ is said to be a convex IVF if for any two vectors $x_1$ and $x_2$ in $\mathcal{X}$,
\[
\textbf{F}(\lambda_1 x_1+\lambda_2 x_2)\preceq
\lambda_1\odot\textbf{F}(x_1)\oplus\lambda_2\odot\textbf{F}(x_2)
\]
for all $\lambda_1,~\lambda_2\in[0,\ 1]$ with $\lambda_1+\lambda_2=1$.
\end{dfn}
\begin{lem}\label{lc1}(See \cite{Wu2007}).
$\textbf{F}$ is convex if and only if $\underline{f}$
and $\overline{f}$ are convex.
\end{lem}
\begin{dfn}(\emph{$gH$-continuity} \cite{Ghosh2016newton}).
Let $\textbf{F}$ be an IVF on a nonempty subset $\mathcal{X}$ of $\mathbb{R}^n$. Let $\bar{x}$ be an interior point of $\mathcal{X}$
and $d\in\mathbb{R}^n$ be such that $\bar{x}+d\in\mathcal{X}$. The function
$\textbf{F}$ is said to be a $gH$-continuous at $\bar{x}$ if
\[
\lim_{\lVert d \rVert\rightarrow 0}\left(\textbf{F}(\bar{x}+d)\ominus_{gH}\textbf{F}(\bar{x})\right)=\textbf{0}.
\]
\end{dfn}
\begin{dfn}(\emph{$gH$-derivative} \cite{Chalco2013kkt}).
Let $\mathcal{X}\subseteq \mathbb{R}$. The $gH$-derivative of an IVF $\textbf{F}:\mathcal{X} \rightarrow I(\mathbb{R})$ at $\bar{x}\in \mathcal{X}$ is defined by
\[
\textbf{F}'(\bar{x})=\displaystyle\lim_{d\rightarrow 0} \frac{\textbf{F}(\bar{x}+d) \ominus_{gH} \textbf{F}(\bar{x})}{d},~~\text{provided the limit exists.}
\]
\end{dfn}
\begin{dfn}\label{pdgh}(\emph{$gH$-partial derivative} \cite{Ghosh2016newton}).
Let $\textbf{F}:\mathcal{X} \rightarrow I(\mathbb{R})$ be an IVF, where $\mathcal{X}$ is a nonempty subset of $\mathbb{R}^n$. We define a function $\textbf{G}_i$ by
\[
\textbf{G}_i (x_i) = \textbf{F} (\bar{x}_1, \bar{x}_2, \ldots, \bar{x}_{i-1}, x_i, \bar{x}_{i+1}, \ldots, \bar{x}_n),
\]
where $\bar{x} = (\bar{x}_1,\, \bar{x}_2,\, \ldots,\, \bar{x}_n)^T\in\mathcal{X}$. If the $gH$-derivative of $\textbf{G}_i$ exists at $\bar{x}_i$, then the $i$-th  $gH$-partial derivative of $\textbf{F}$ at $\bar{x}$, denoted $D_i \textbf{F}(\bar{x})$, is defined by
\[
D_i \textbf{F}(\bar{x})=\textbf{G}'_i (\bar{x}_i)~~\text{for all}~~i = 1,\, 2,\, \ldots,\, n.
\]
\end{dfn}
\begin{dfn} (\emph{$gH$-gradient} \cite{Ghosh2016newton}).
Let $\mathcal{X}$ be a nonempty subset of $\mathbb{R}^n$. The $gH$-gradient of an IVF $\textbf{F}:\mathcal{X} \rightarrow I(\mathbb{R})$ at a point $\bar{x} \in \mathcal{X}$, denoted $\nabla \textbf{F} (\bar{x})$, is defined by
\[
\nabla \textbf{F} (\bar{x})=\left(   D_1\textbf{F}(\bar{x}),\,
         D_2\textbf{F}(\bar{x}),\, \ldots,\,
         D_n\textbf{F}(\bar{x})
\right)^T.
\]
\end{dfn}
\begin{dfn}(\emph{Linear IVF} \cite{Ghosh2019gradient}). \label{dlivf}
Let $\mathcal{X}$ be a linear subspace of $\mathbb{R}^n$. The function $\textbf{F}: \mathcal{X} \rightarrow I(\mathbb{R})$ is said to be linear if
$\textbf{F}(x)=\bigoplus_{i=1}^n x_i\odot\textbf{F}(e_i)~ \text{for all}~x=(x_1, x_2,\ldots, x_n)^T\in \mathcal{X}$, where $e_i$ is the $i$-th standard basis vector of $\mathbb{R}^n$, $i = 1, 2, \ldots, n$ and
`$\bigoplus_{i=1}^n$' denotes successive addition of $n$ number of intervals.
\end{dfn}
\begin{dfn}\label{dghd} (\emph{$gH$-differentiability} \cite{Ghosh2019gradient}).
Let $\mathcal{X}$ be a nonempty subset of $\mathbb{R}^n$. An IVF $\textbf{F}:\mathcal{X} \rightarrow I(\mathbb{R})$ is said to be $gH$-differentiable at a point $\bar{x} \in \mathcal{X}$ if there exist a linear IVF $\textbf{L}_{\bar{x}}:\mathbb{R}^n\rightarrow I(\mathbb{R})$, an IVF $\textbf{E}(\textbf{F}(\bar{x});d)$ and a $\delta~>~0$ such that
\[
\left(\textbf{F}(\bar{x}+d)\ominus_{gH} \textbf{F}(\bar{x})\right)\ominus_{gH} \textbf{L}_{\bar{x}}(d)=\lVert d \rVert \odot \textbf{E}(\textbf{F}(\bar{x});d)~~\text{for all}~~d~~\text{such that}~~\lVert d \rVert~<~ \delta,
\]
where $\textbf{E}(\textbf{F}(\bar{x});d)\rightarrow \textbf{0}$ as $\lVert d \rVert\rightarrow 0$.
\end{dfn}
\begin{thm}\emph{(See \cite{Ghosh2019gradient})}.\label{td2}
Let an IVF $\textbf{F}$ on a nonempty open convex subset $\mathcal{X}$ of $\mathbb{R}^n$ be $gH$-differentiable at $x\in\mathcal{X}$. If the function $\textbf{F}$ is convex on $\mathcal{X}$, then
\[
(y-x)^T \odot \nabla\textbf{F}(x)~\preceq~ \textbf{F}(y)\ominus_{gH}\textbf{F}(x) ~ \text{ for all } x,~y\in \mathcal{X}.
\]
\end{thm}
%
%
%


\section{$gH$-subgradient Method}\label{ssm}
\noindent In this section, we illustrate a technique to obtain the solutions to the following IOP:
\begin{equation}\label{IOP}
\displaystyle \min_{x \in \mathcal{X}} \textbf{F}(x),
\end{equation}
where $\textbf{F}:\mathcal{X}\rightarrow I(\mathbb{R})$ is a convex IVF on the nonempty convex subset $\mathcal{X}$ of $\mathbb{R}^n$. In order to develop the technique, we go through the concepts of $gH$-subgradient for convex IVFs, efficient solution and nondominated solution to the IOP (\ref{IOP}), and efficient directions of an IVF. Further, considering a numerical problem regarding convex IOP, we apply the proposed method to capture the efficient solutions to the considered IOP and analyze each step.
 \\

Since an IVF is a special case of a fuzzy-valued function (FVF), here we adopt the definition of $gH$-subgradient for FVFs \cite{Hai2018} to define the $gH$-subgradient of IVFs.
\begin{dfn}\label{dsg}(\emph{gH-subgradient}). Let $\mathcal{X}$ be a nonempty open subset of $\mathbb{R}^n$ and $\textbf{F}: \mathcal{X} \rightarrow I(\mathbb{R})$ be a convex IVF. Then an element $\widehat{\textbf{G}}=(\textbf{G}_1, \textbf{G}_2,\ldots,\textbf{G}_n)\in I(\mathbb{R})^n$ is said to be a $gH$-sugradient of $\textbf{F}$ at $\bar{x}$ if
\begin{equation}\label{esg}
(x-\bar{x})^T\odot\widehat{\textbf{G}}\preceq \textbf{F}(x)\ominus_{gH} \textbf{F}(\bar{x})
\end{equation}
and then we write $\widehat{\textbf{G}}\in\partial\textbf{F}(\bar{x})$.\\

The set $\partial \textbf{F}(\bar{x})$ of all $gH$-subgradients of the convex IVF $\textbf{F}$ at $\bar{x}\in \mathcal{X}$ is called $gH$-subdifferential of $\textbf{F}$ at $\bar{x}$.
\end{dfn}
\begin{note}\label{nsg1}
In view of Theorem \ref{td2}, it is to be noted that if $\textbf{F}$ is $gH$-differentiable at $\bar{x}\in \mathcal{X}$, then $\nabla\textbf{F}(\bar{x}) \in \partial \textbf{F}(\bar{x})$.
\end{note}
\begin{dfn}
(\emph{Efficient solution} \cite{Ghosh2019derivative}). A point $\bar{x} \in \mathcal{X}$ is called an efficient solution of the IOP (\ref{IOP}) if $\textbf{F}(x) \nprec \textbf{F} (\bar{x})$ for all $x \in
\mathcal{X}$.
\end{dfn}
\begin{dfn}
(\emph{Nondominated solution}). If $\bar{x} \in \mathcal{X}$ is an efficient solution of the IOP (\ref{IOP}), then $\textbf{F} (\bar{x})$ is said to be a nondominated solution of the IOP (\ref{IOP}).
\end{dfn}
\begin{dfn}\label{ded}
(\emph{Efficient-direction} \cite{Ghosh2019gradient}). Let $\mathcal{X}\subseteq\mathbb{R}^n$. A direction $d\in \mathbb{R}^n$ is said to be an efficient-direction of an IVF $\textbf{F}:\mathcal{X}\rightarrow I(\mathbb{R})$ at $\bar{x} \in \mathcal{X}$ if there exists a $\delta~>~0$ such that
\begin{enumerate}[(i)]
\item\label{ced1} $\textbf{F} (\bar{x}) ~\npreceq~ \textbf{F}(\bar{x}+\lambda d) ~\text{for all}~\lambda \in (0,\delta),$

\item\label{ced2} there also exists a point $x'=\bar{x}+\alpha d$ with $\alpha \in (0,\delta)$ and a positive real number $\delta'~\leq~ \alpha$ such that
     \[
      \textbf{F} (x'+\lambda d) ~\nprec~ \textbf{F}(x')~\text{for all}~\lambda \in (-\delta',\delta').
     \]
     The point $x'$ is known as an efficient solution of $\textbf{F}$ in the direction $d$.
\end{enumerate}
\end{dfn}
\begin{thm}\emph{ (See \cite{Ghosh2019gradient})}.\label{tsgmed}
Let $\mathcal{X}\subseteq\mathbb{R}^n$ be a nonempty set and $\textbf{F}:\mathcal{X}\rightarrow I(\mathbb{R})$ be $gH$-differentiable at a point $\bar{x}\in\mathcal{X}$. Then, the direction $-\mathcal{W}(\nabla\textbf{F}(\bar{x}))$ is an efficient-direction of $\textbf{F}$ at $\bar{x}$, provided $0\not\in D_i\textbf{F}(\bar{x})$ for at least one $i\in \{1, 2, \ldots, n\}$, where for any $w,~w'\in[0,1]$ with $w+w'=1$, the map $\mathcal{W}: I(\mathbb{R})^n\rightarrow \mathbb{R}^n$ is defined by
\begin{equation}\label{wmap}
\mathcal{W}(\textbf{A}_1, \textbf{A}_2, \ldots, \textbf{A}_n)=(w\underline{a}_1+w'\overline{a}_1, w\underline{a}_2+w'\overline{a}_2, \ldots, w\underline{a}_n+w'\overline{a}_n)^T.
\end{equation}
\end{thm}

Based on Theorem \ref{tsgmed}, we develop the method to find the efficient solutions to the IOP (\ref{IOP}), where the objective function $\textbf{F}$ may not be $gH$-differentiable at each point in the feasible region. The searching direction at each point $\bar{x}\in \mathcal{X}$ of the proposed method is $-\mathcal{W}(\widehat{\textbf{G}})$, where $\widehat{\textbf{G}}\in \partial \textbf{F}(\bar{x})$. Since, in the proposed method, to generate a direction we use  $gH$-subgradient of the objective function $\textbf{F}$, we name the method as \emph{$gH$-subgradient method}. Similar to the subgradient method for conventional optimization problems \cite{Bertsekas1999,Boyd2004}  the proposed $gH$-subgradient method has no stoping criteria. In the $gH$-subgradient method, we consider the diminishing step length $\alpha_k$ at each iteration $k$, i.e,
\[
 \alpha_{k}>0,~~\lim_{k\rightarrow\infty}\alpha_k=0~~\text{and}~~\sum_{k=1}^\infty\alpha_k=\infty.
\]
The algorithmic implementation of the proposed $gH$-subgradient method is presented in Algorithm \ref{algosg}. It is noteworthy that for the degenerate case of the IVF $\textbf{F}$, i.e., for $\underline{f}(x)=\overline{f}(x)=f(x)$ for all $x\in\mathcal{X}$, the Algorithm \ref{algosg} reduces to the conventional subgradient method with diminishing step size.
\begin{algorithm}[H]
    \caption{$gH$-subgradient method for IOPs}\label{algosg}
    \begin{algorithmic}[1]
      \REQUIRE Given the initial point $x_{0}$, the IVF $\textbf{F}:\mathcal{X}(\subseteq\mathbb{R}^n)\rightarrow I(\mathbb{R})$, the maximum number of iterations $m$, and $w, w' \in [0, 1]$ such that $w+w'=1$.
      \STATE Set $k = 1$, $\mathscr{E}_s=\{x_k\}$, and $\mathscr{N}_s=\left\{\textbf{F}(x_k)\right\}$.
      \STATE \textbf{while} $k\leq m$
      \STATE ~~~~~~~~~~Find a $\widehat{\textbf{G}}_{k}\in\partial \textbf{F}(x_k)$ such that $\mathcal{W}(\widehat{\textbf{G}}_{k})\neq 0$ and a sequence $\{\alpha_{k}\}$ such that
            \[
            ~~~~~~~~~~\alpha_{k}>0,~~\lim_{k\rightarrow\infty}\alpha_k=0~~\text{and}~~\sum_{k=1}^\infty\alpha_k=\infty.
            \]
      \STATE ~~~~~~~~~~Calculate
             \[
             ~~~~~~~~~~x_{k+1} = x_{k} - \alpha_{k}\mathcal{W}(\widehat{\textbf{G}}_{k}).
             \]
      \STATE ~~~~~~~~~~Set $\widetilde{\mathscr{E}}_s=\left\{x\in\mathscr{E}_s\mid \textbf{F}(x_{k+1})\prec\textbf{F}(x)\right\}$ and $\widetilde{\mathscr{N}}_s=\left\{\textbf{A}\in\mathscr{N}_s\mid \textbf{F}(x_{k+1})\preceq\textbf{A}\right\}$.
      \STATE ~~~~~~~~~~Set $\mathscr{E}_s\leftarrow \mathscr{E}_s\setminus\widetilde{\mathscr{E}}_s$.
      \STATE ~~~~~~~~~~Set $\mathscr{N}_s\leftarrow \mathscr{N}_s\setminus\widetilde{\mathscr{N}}_s$.
      \STATE ~~~~~~~~~~\textbf{if} $\textbf{F}(x)\nprec \textbf{F}(x_{k+1})$ for all $x\in\mathscr{E}_s$,
      \STATE ~~~~~~~~~~~~~~~~~~~~$\mathscr{E}_s\leftarrow\mathscr{E}_s\cup\left\{x_{k+1}\right\}$
      \STATE ~~~~~~~~~~\textbf{end if}
      \STATE ~~~~~~~~~~\textbf{if} $\textbf{A}\npreceq \textbf{F}(x_{k+1})$ for all $\textbf{A}\in\mathscr{N}_s$,
      \STATE ~~~~~~~~~~~~~~~~~~~~$\mathscr{N}_s\leftarrow\mathscr{N}_s\cup\left\{\textbf{F}(x_{k+1})\right\}$
      \STATE ~~~~~~~~~~\textbf{end if}
      \STATE Set $k \leftarrow k + 1$.
      \STATE \textbf{end while}
      \STATE \textbf{Return} $\mathscr{E}_s$: the set of efficient solutions and $\mathscr{N}_s$: the set of non-dominated solutions.
    \end{algorithmic}
  \end{algorithm}
\begin{note}
It is to be mentioned that although we named the set $\mathscr{E}_s$ of Algorithm \ref{algosg} as the efficient solution set of the IOP (\ref{IOP}), the elements of the set $\mathscr{E}_s$ may not be the actual efficient solutions to the IOP (\ref{IOP}). But similar to the subgradient method for the conventional optimization problems, the elements of the set $\mathscr{E}_s$ will be the best elements among all the points $x_k$ of $\mathcal{X}$ generated by Algorithm \ref{algosg}. Similarly, the elements of the set $\mathscr{N}_s$ will be the best elements among all the intervals $\textbf{F}(x_k)$ of $\textbf{F}(\mathcal{X})$ generated by Algorithm \ref{algosg}.
\end{note}
\begin{example}\label{exsgm}
Let us consider the following IOP:
\begin{equation}\label{exsgmiop}
\min_{x\in \mathcal{X}} \textbf{F}(x)=
\begin{cases}
[3, 7] \ominus_{gH} [-1, 0]\odot \lvert x \rvert, & \text{if } -1 ~\leq~ x ~\leq~ 1 \\
[3, 5] \oplus [1, 2]\odot \lvert x \rvert,  & \text{otherwise}
\end{cases}\end{equation}
where $\mathcal{X}=[-3, 3]$.\\
The graph of the IVF $\textbf{F}$ is depicted by the gray shaded region in Figure \ref{fexsgm}. From Figure \ref{fexsgm}, it is clear that there does not exist any $x(\neq 0) \in \mathcal{X}$ such that $\textbf{F}(x)\prec\textbf{F}(0)$, however $\textbf{F}(0)\prec\textbf{F}(x)$ for all $x(\neq 0) \in \mathcal{X}$. Hence, $x=0$ is the only efficient solution of the IOP (\ref{exsgmiop}).
\begin{figure}[H]
\begin{center}
\includegraphics[scale=0.8]{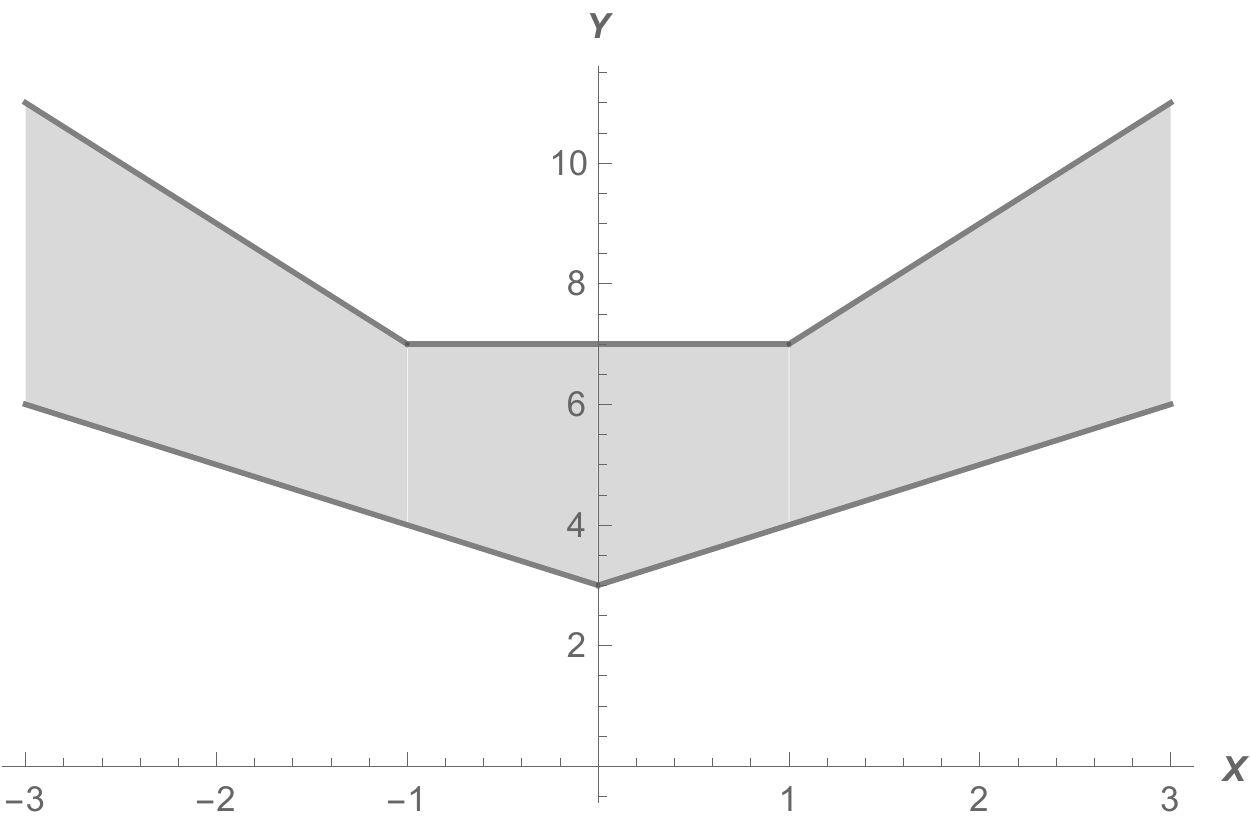}
    \caption{Interval-valued function $\textbf{F}$ of the IOP (\ref{exsgmiop}) is illustrated by gray shaded region.}\label{fexsgm}
\end{center}
\end{figure}
At $\bar{x}=-1$,
\begingroup\allowdisplaybreaks\begin{align*}
\left(\textbf{F}'(\bar{x})\right)_+:=~&\displaystyle\lim_{d\rightarrow 0+} \frac{1}{d}\odot \left(\textbf{F}(\bar{x}+d) \ominus_{gH} \textbf{F}(\bar{x})\right)\\
=~&\displaystyle\lim_{d\rightarrow 0+} \frac{1}{d}\odot \left(\textbf{F}(d-1) \ominus_{gH} \textbf{F}(-1)\right)\\
=~&\displaystyle\lim_{d\rightarrow 0+} \frac{1}{d}\odot \big(\left([3, 7] \ominus_{gH} [-1, 0]\odot \lvert d-1 \rvert\right)\ominus_{gH} [4, 7]\big)\\
=~&\displaystyle\lim_{d\rightarrow 0+} \frac{1}{d}\odot \left([4-d, 7]\ominus_{gH} [4, 7]\right)\\
=~&\displaystyle\lim_{d\rightarrow 0+} \frac{1}{d}\odot [-d, 0]\\
=~&[-1, 0],
\end{align*}\endgroup
and
\begingroup\allowdisplaybreaks\begin{align*}
\left(\textbf{F}'(\bar{x})\right)_-:=~&\displaystyle\lim_{d\rightarrow 0-} \frac{1}{d}\odot \left(\textbf{F}(\bar{x}+d) \ominus_{gH} \textbf{F}(\bar{x})\right)\\
=~&\displaystyle\lim_{d\rightarrow 0-} \frac{1}{d}\odot \left(\textbf{F}(d-1) \ominus_{gH} \textbf{F}(-1)\right)\\
=~&\displaystyle\lim_{d\rightarrow 0-} \frac{1}{d}\odot \big(\left([3, 5] \oplus [1, 2]\odot \lvert d-1 \rvert\right)\ominus_{gH} [4, 7]\big)\\
=~&\displaystyle\lim_{d\rightarrow 0-} \frac{1}{d}\odot \left([4-d, 7-2d]\ominus_{gH} [4, 7]\right)\\
=~&\displaystyle\lim_{d\rightarrow 0-} \frac{1}{d}\odot [-d, -2d]\\
=~&[-2, -1]\\
\neq~&\left(\textbf{F}'(\bar{x})\right)_+.
\end{align*}\endgroup
Therefore, at the point $x=-1$, $\textbf{F}$ is not $gH$-differentiable. Similarly it can be proved that $\textbf{F}$ is not $gH$-differentiable at the points $x=0$ and $x=1$.\\

\noindent Next, we apply Algorithm \ref{algosg} on the IOP (\ref{exsgmiop}) and try to capture the efficient solutions. Although we can start the iterations from any point of $\mathcal{X}$, to analyze the $gH$-subgradient method properly, we start the iterations from $x=-1$ because $\textbf{F}$ is not $gH$-differentiable at $x=-1$.\\

\noindent Let us consider $w=\tfrac{2}{3}$, $w'=\tfrac{1}{3}$ and the step length $\alpha_k = \tfrac{1}{k}$ at each iteration $k$. Therefore,
\[
x_1=-1,~~\alpha_1=1,~~\mathscr{E}_s=\{-1\}~~\text{and}~~\mathscr{N}_s=\{[4, 7]\}.
\]
Let $\textbf{G}_1\in \partial\textbf{F}(x_1)$. Hence, for all $d\in[0, 2]$ we have
\begingroup\allowdisplaybreaks\begin{align*}
\textbf{G}_1\odot d\preceq~&\textbf{F}(d-1)\ominus_{gH} \textbf{F}(-1)\\
=~&\left([3, 7] \ominus_{gH} [-1, 0]\odot \lvert d-1 \rvert\right)\ominus_{gH} [4, 7]\\
=~&\begin{cases}
[-d, 0] & \text{for}~~0\leq d\leq 1\\
[d-2, 0] & \text{for}~~1\leq d\leq 2.\\
\end{cases}
\end{align*}\endgroup
Thus,
\begin{equation}\label{esss1}
\underline{g}_1\leq -1~~\text{and}~~\overline{g}_1\leq 0.
\end{equation}
Further, for all $d\in [-2, 0]\cup[2, 4]$, we have
\begingroup\allowdisplaybreaks\begin{align*}
\textbf{G}_1\odot d\preceq~&\textbf{F}(d-1)\ominus_{gH} \textbf{F}(-1)\\
=~&\left([3, 5] \oplus [1, 2]\odot \lvert d-1 \rvert\right) \ominus_{gH} [4, 7]\\
=~&\begin{cases}
[-d, -2d] & \text{for}~~-2\leq d\leq 0\\
[d-2, 2d-4] & \text{for}~~2\leq d\leq 4.\\
\end{cases}
\end{align*}\endgroup
Therefore,
\begin{equation}\label{esss2}
-2\leq\underline{g}_1\leq 0~~\text{and}~~-1\leq\overline{g}_1\leq 0.
\end{equation}
By the inequalities (\ref{esss1}) and (\ref{esss2}), we obtain
\[
[-2, -1]\preceq\textbf{G}_1\preceq [-1, 0].
\]
We choose $\textbf{G}_1=\left[-\tfrac{3}{2}, -\tfrac{1}{2}\right]$. Hence, $\mathcal{W}(\textbf{G}_1)=-\tfrac{7}{6}$. Therefore,
\[
x_2=x_1-\alpha_1\mathcal{W}(\textbf{G}_1)=-1+\tfrac{7}{6}=\tfrac{1}{6}.
\]
Since $\textbf{F}(\tfrac{1}{6})=\left[\tfrac{19}{6}, 7\right]\prec[4, 7]=\textbf{F}(-1)$,
\[
\widetilde{\mathscr{E}}_s=\{-1\}~~\text{and}~~\widetilde{\mathscr{N}}_s=\{[4, 7]\}.
\]
Therefore, we delete $-1$ and $[4, 7]$ from the sets $\mathscr{E}_s$ and $\mathscr{N}_s$, respectively and both the sets $\mathscr{E}_s$ and $\mathscr{N}_s$ are empty now. So, according to Algorithm \ref{algosg}, there does not exist any $x\in\mathscr{E}_s$ such that $\textbf{F}(x)\prec\textbf{F}(\tfrac{1}{6})$ and hence, the new
\[
\mathscr{E}_s=\left\{\tfrac{1}{6}\right\}~~\text{and}~~\mathscr{N}_s=\left\{\textbf{F}(\tfrac{1}{6})\right\}=\left\{\left[\tfrac{19}{6}, 7\right]\right\}.
\]
Further, it can be easily checked that $\textbf{F}$ is $gH$-differentiable at $x=\tfrac{1}{6}$. Thus, due to Note \ref{nsg1}, we get
\[
\textbf{G}_2=\nabla\textbf{F}(\tfrac{1}{6})=[0, 1]\in \partial\textbf{F}(x_2)
\]
and $\mathcal{W}(\textbf{G}_2)=\tfrac{1}{3}$. As $\alpha_2=\tfrac{1}{2}$,
\[
x_3=x_2-\alpha_2\mathcal{W}(\textbf{G}_2)=\tfrac{1}{6}-\tfrac{1}{2}\times\tfrac{1}{3}=0.
\]
Since $\textbf{F}(0)=[3, 7]\prec\left[\tfrac{19}{6}, 7\right]=\textbf{F}(\tfrac{1}{6})$,
\[
\widetilde{\mathscr{E}}_s=\left\{\tfrac{1}{6}\right\}~~\text{and}~~\widetilde{\mathscr{N}}_s=\left\{\left[\tfrac{19}{6}, 7\right]\right\}.
\]
Hence, due to Algorithm \ref{algosg}, we delete $-1$ and $[4, 7]$ from the sets $\mathscr{E}_s$ and $\mathscr{N}_s$, respectively and both the sets $\mathscr{E}_s$ and $\mathscr{N}_s$ become empty. Therefore, there does not exist any $x\in\mathscr{E}_s$ such that $\textbf{F}(x)\prec\textbf{F}(0)$ and hence, the new
\[
\mathscr{E}_s=\left\{0\right\}~~\text{and}~~\mathscr{N}_s=\left\{\textbf{F}(0)\right\}=\left\{\left[3, 7\right]\right\}.
\]
As $x=0$ is the only efficient solution of the IOP (\ref{exsgmiop}), if we perform more iterations, none of the generated $\textbf{F}(x_k)$ can dominate $\textbf{F}(0)$. Therefore, neither any $x_k$ can take entry in $\mathscr{E}_s=\{0\}$ nor any $\textbf{F}(x_k)$ can take entry in $\mathscr{N}_s=\{[3, 7]\}$ for $k\geq2$.
\end{example}

\begin{thm}[Convergence of $gH$-subgardient method]\label{tcsm}
Let $\textbf{F}:\mathcal{X}\rightarrow I(\mathbb{R})$ be a convex IVF, where $\mathcal{X}\subset\mathbb{R}^n$, and let the nonempty set $\Omega$ be the set of efficient solution of $\textbf{F}$. Given $x_1\in \mathcal{X}$, the sequence $\{x_k\}$ is generated by Algorithm \ref{algosg}, either $\{x_k\}\in \Omega$ for all $k$ or there exists a subsequence $\{x_{k_n}\}$ of $\{x_k\}$ such that $\textbf{F}(x_{k_{n+1}})\preceq \textbf{F}(x_{k_n})$ and $\lim\limits_{n\to \infty}\textbf{F}(x_{k_n})=\textbf{F}(x^*),$ where $x^*\in \Omega$.
\end{thm}

\begin{proof}
For the given sequence $\{x_k\}$, let there does not exist any subsequence $\{x_{k_n}\}$ of $\{x_k\}$ such that $\textbf{F}(x_{k_{n+1}})\preceq \textbf{F}(x_{k_n})$, then $\{x_k\}\in \Omega$.\\

\noindent
Next, we suppose that $x_k \notin \Omega$ for all $k$. Then, there exists a subsequence $\{x_{k_n}\}$ of $\{x_k\}$ such that \[\textbf{F}(x_{k_{n+1}})\preceq \textbf{F}(x_{k_n}).\]
\noindent
This implies that \[\textbf{F}(x_{k_1})\succeq \textbf{F}(x_{k_2})\succeq \textbf{F}(x_{k_3})\succeq\cdots\succeq  \textbf{F}(x^*).\]
Since
\begingroup\allowdisplaybreaks\begin{align}\label{hhhh}
&\lVert \textbf{F}(x_{k_n}) \ominus_{gH} \textbf{F}(x^*) \rVert_{I(\mathbb{R})}\leq  \lVert \textbf{F}(x_{k_{n-1}}) \ominus_{gH} \textbf{F}(x^*) \rVert_{I(\mathbb{R})} \nonumber\\
\Longrightarrow~& \lVert \textbf{F}(x_{k_n}) \ominus_{gH} \textbf{F}(x^*) \rVert_{I(\mathbb{R})}=\epsilon_1  \lVert \textbf{F}(x_{k_{n-1}}) \ominus_{gH} \textbf{F}(x^*) \rVert_{I(\mathbb{R})},~\text{where}~\epsilon_1 \in (0, 1),
\end{align}\endgroup

and
\begingroup\allowdisplaybreaks\begin{align}\label{kkkk}
&\lVert \textbf{F}(x_{k_{n-1}}) \ominus_{gH} \textbf{F}(x^*) \rVert_{I(\mathbb{R})}\leq  \lVert \textbf{F}(x_{k_{n-2}}) \ominus_{gH} \textbf{F}(x^*) \rVert_{I(\mathbb{R})} \nonumber\\
\Longrightarrow~& \lVert \textbf{F}(x_{k_{n-1}}) \ominus_{gH} \textbf{F}(x^*) \rVert_{I(\mathbb{R})}=\epsilon_2  \lVert \textbf{F}(x_{k_{n-2}}) \ominus_{gH} \textbf{F}(x^*) \rVert_{I(\mathbb{R})},~\text{where}~\epsilon_2 \in (0, 1),
\end{align}\endgroup
 we have
\begingroup\allowdisplaybreaks\begin{align}\label{k}
&\lVert \textbf{F}(x_{k_{n}}) \ominus_{gH} \textbf{F}(x^*) \rVert_{I(\mathbb{R})}= \epsilon_1 \epsilon_2 \lVert \textbf{F}(x_{k_{n-2}}) \ominus_{gH} \textbf{F}(x^*) \rVert_{I(\mathbb{R})} \nonumber\\
\Longrightarrow~& \lVert \textbf{F}(x_{k_{n-1}}) \ominus_{gH} \textbf{F}(x^*) \rVert_{I(\mathbb{R})}\leq \epsilon^2  \lVert \textbf{F}(x_{k_{n-2}}) \ominus_{gH} \textbf{F}(x^*) \rVert_{I(\mathbb{R})},~\text{where}~\epsilon =\max\{\epsilon_1, \epsilon_2\}.
\end{align}\endgroup
By the same process, we obtain
\[\lVert \textbf{F}(x_{k_{n}}) \ominus_{gH} \textbf{F}(x^*) \rVert_{I(\mathbb{R})}= \epsilon^{n-1} \lVert \textbf{F}(x_{k_{1}}) \ominus_{gH} \textbf{F}(x^*) \rVert_{I(\mathbb{R})}.\]
\noindent
Since $\lVert \textbf{F}(x_{k_{1}}) \ominus_{gH} \textbf{F}(x^*) \rVert_{I(\mathbb{R})}$ is finite, taking limit $n\to \infty$, we have
\begingroup\allowdisplaybreaks\begin{align*}
&\lVert \textbf{F}(x_{k_{n}}) \ominus_{gH} \textbf{F}(x^*) \rVert_{I(\mathbb{R})}\to 0\\
\Longrightarrow~& \lim\limits_{n\to \infty}\textbf{F}(x_{k_n})=\textbf{F}(x^*).
\end{align*}\endgroup
\end{proof}

\section{Application}\label{sa}

\noindent  In this section, we apply the proposed $gH$-subgradient method in solving the $\ell_1$ penalized linear regression problems, which is known as \emph{lasso problems} for interval-valued data.\\

Suppose a set of  $n$ pairs of data $(\textbf{X}_k, \textbf{Y}_k)$ is given, where $\textbf{Y}_k\in I(\mathbb{R})$ is the corresponding interval-valued output of interval-valued features $\textbf{X}_k\in I(\mathbb{R})^l$ for all $k\in\{1, 2, \ldots, n\}$. Our aim is fitting a function $\mathcal{H}\left(\cdot ;\beta\right):I(\mathbb{R})^l\to I(\mathbb{R})$, defined by
\[
\mathcal{H}\left(\textbf{X};\beta\right)=\bigoplus_{i=1}^l\beta_i\odot\textbf{X}^i,
\]
where $\beta\in \mathbb{R}^l$ is a parameter vector such that $\mathcal{H}\left(\textbf{X}_k;\hat{\beta}\right)$ will be one of the best approximations of $\textbf{Y}_k$ for all $k\in\{1, 2, \ldots, n\}$. By `one of the best approximations' we mean that $\mathcal{H}\left(\textbf{X} ;\beta\right)$ gives a nondominated error. Evidently, if $\hat{\beta}$ is an efficient solution of the following IOP:
\begin{equation}\label{eef}
\min_{\beta\in\mathbb{R}^l} \boldsymbol{\mathcal{E}}  (\beta)=\boldsymbol{\mathcal{E}}_1(\beta)\oplus\boldsymbol{\mathcal{E}}_2(\beta),
\end{equation}
where
\[
\boldsymbol{\mathcal{E}}_1(\beta)=\frac{1}{2}\bigoplus_{k=1}^n\big(\left(\mathcal{H}\left(\textbf{X}_k;\beta\right)\ominus_{gH}\textbf{Y}_k\right)\odot_s\left(\mathcal{H}\left(\textbf{X}_k;\beta\right)\ominus_{gH}\textbf{Y}_k\right)\big)
\]
and
\[
\boldsymbol{\mathcal{E}}_2(\beta) = \textbf{L} \odot \left(\lvert\beta_1\rvert+\lvert\beta_2\rvert+\ldots+\lvert\beta_l\rvert\right);
\]
$\textbf{L}\in I(\mathbb{R}_+)$ is interval-valued tuning parameter, then $\mathcal{H}\left(\textbf{X};\hat{\beta}\right)$ can be considered as an efficient choice of the approximating function $\mathcal{H}\left(\textbf{X};\beta\right)$.\\

\noindent It is observable that the functions $\mathcal{H}\left(\textbf{X}_k; \cdot\right)$ and $\boldsymbol{\mathcal{E}}_2$ are convex IVFs from $\mathbb{R}^l$ to $I(\mathbb{R})$. Since, to construct the function $\boldsymbol{\mathcal{E}}_1$ we have used the special product $\odot_s$ instead of $\odot$, the function $\boldsymbol{\mathcal{E}}_1$ is always a convex IVF from $\mathbb{R}^l$ to $I(\mathbb{R})$ for all $\textbf{X}_k$ and $\textbf{Y}_k$. Thus, the error function $\textbf{E}$ is a convex IVF from $\mathbb{R}^l$ to $I(\mathbb{R})$ for all $\textbf{X}_k$ and $\textbf{Y}_k$. \\
%
%
%

\noindent
We note that
\begingroup\allowdisplaybreaks\begin{align*}
\partial\boldsymbol{\mathcal{E}} (\beta)~=&~\partial\boldsymbol{\mathcal{E}}_1(\beta)\oplus\partial\boldsymbol{\mathcal{E}}_2(\beta)\\
=&~\left\{\nabla\boldsymbol{\mathcal{E}}_1(\beta)\right\}\oplus\partial\boldsymbol{\mathcal{E}}_2(\beta)\\
=&~\left\{\nabla\boldsymbol{\mathcal{E}}_1(\beta)\right\}\oplus\bigoplus_{i=1}^l\partial\left(\textbf{L}\odot\lvert\beta_i\rvert\right).
\end{align*}\endgroup
Therefore, one possible choice of $gH$-subgradient $\widehat{\textbf{G}}=\left(\textbf{G}_{1}, \textbf{G}_{2}, \ldots, \textbf{G}_{l}\right)\in\partial\boldsymbol{\mathcal{E}} (\beta)$ is given by
\begingroup\allowdisplaybreaks\begin{align*}
\textbf{G}_i~=&~\begin{cases}
D_i\boldsymbol{\mathcal{E}}_1(\beta)\oplus\textbf{L} & \text{if}~~ \beta_i\geq 0\\
D_i\boldsymbol{\mathcal{E}}_1(\beta)\oplus\textbf{L}\odot (-1) & \text{if}~~ \beta_i< 0\\
\end{cases}\\
~=&~\begin{cases}
\bigoplus_{k=1}^n\left(\bigoplus_{i=1}^l\beta_i\odot\textbf{X}_k^i\ominus_{gH}\textbf{Y}_k\right)\odot_s \textbf{X}_k^i\oplus\textbf{L} & \text{if}~~ \beta_i\geq 0\\
\bigoplus_{k=1}^n\left(\bigoplus_{i=1}^l\beta_i\odot\textbf{X}_k^i\ominus_{gH}\textbf{Y}_k\right)\odot_s \textbf{X}_k^i\oplus\textbf{L}\odot (-1) & \text{if}~~ \beta_i< 0\\
\end{cases}
\end{align*}\endgroup
for all $i\in\{1, 2, \ldots, l\}$.
Hence, by applying the $gH$-subgradient method on the IOP (\ref{eef}) one can obtain efficient parameter vector $\hat{\beta}$ for the function $\mathcal{H}\left(\cdot; \beta\right)$.\\

\noindent For instance, we consider the two-dimensional interval-valued features with their interval-valued output as displayed in Table \ref{tablesgmdata}. \\

\noindent
Considering the interval-valued tuning parameter $\textbf{L} = [0.03, 0.06]$ and applying Algorithm \ref{algosg} with diminishing step length $\alpha_k=\tfrac{7}{k+100000}$ at $k$-th iteration on the IOP (\ref{eef}) corresponding to the function $\mathcal{H}^2\left(\cdot ;\beta\right)$, after 10000 iterations in Matlab R2015a platform with Intel Core i5-2430M, 2:40 GHz CPU, 3 GB RAM, 32-bit Windows 7 environment, we obtain the efficient solution set as well as non-dominated solution set of IOP (\ref{eef})  as given in Table \ref{tablesolsgm} for four different values of $w$ with two different initial points.\\

\noindent The comparison of actual interval-valued outputs $\textbf{Y}_k$ and estimated interval-valued outputs $\mathcal{H}\left(\textbf{X}_k; \hat{\beta}\right)$ of the interval-valued features $\textbf{X}_k=\left(\textbf{X}^1_k, \textbf{X}^2_k\right)$ for different values of $w$ with different initial points are illustrated in Figure \ref{fsgmlasso}. The common portions of $\textbf{Y}_k$ with $\mathcal{H}\left(\textbf{X}_k;\hat{\beta}\right)$ are depicted by blue regions, where as the extended portions of $\textbf{Y}_k$ and $\mathcal{H}\left(\textbf{X}_k ;\hat{\beta}\right)$ are depicted by red and yellow regions, respectively.\\

\begin{longtable}{lll}
\caption{Data for interval-valued lasso problem}\label{tablesgmdata}\\
\toprule
$\textbf{X}^1_k$ & $\textbf{X}^2_k$ & $\textbf{Y}_k$\\
\midrule
\endhead
$[	15.88	,	16.54	]$ & $[	37.28	,	38.04	]$ & $[	398.74	,	 409.02	]$ \\
$[	16.41	,	16.85	]$ & $[	37.84	,	38.4	]$ & $[	405.9	,	 413.5	]$ \\
$[	16.87	,	17.43	]$ & $[	38.48	,	38.97	]$ & $[	413.5	,	 420.95	]$ \\
$[	16.77	,	17.29	]$ & $[	38.3	,	38.97	]$ & $[	411.48	,	 420.39	]$ \\
$[	16.35	,	17	]$ & $[	38.93	,	39.52	]$ & $[	415.47	,	424.18	 ]$ \\
$[	16.5	,	16.84	]$ & $[	39.57	,	40.32	]$ & $[	421.83	,	 430.74	]$ \\
$[	16.33	,	16.77	]$ & $[	39.66	,	40.29	]$ & $[	421.96	,	 430.19	]$ \\
$[	16.82	,	17.09	]$ & $[	39.77	,	40.2	]$ & $[	424.91	,	 430.66	]$ \\
$[	16.54	,	17.13	]$ & $[	39.96	,	40.84	]$ & $[	425.5	,	 436.58	]$ \\
$[	16.71	,	17.31	]$ & $[	40.52	,	41.22	]$ & $[	431.22	,	 440.72	]$ \\
$[	17.13	,	17.62	]$ & $[	40.35	,	41.05	]$ & $[	431.37	,	 440.43	]$ \\
$[	16.03	,	16.59	]$ & $[	41.09	,	41.5	]$ & $[	433.63	,	 440.36	]$ \\
\bottomrule
\end{longtable}
\begin{longtable}{lllll}
\caption{Output of Algorithm \ref{algosg} to find efficient solutions to IOP (\ref{eef})}\label{tablesolsgm}\\
\toprule
$w$ & $w'$ & Initial point & Efficient solution set $\mathscr{E}_s$ & Nondominated solution set $\mathscr{N}_s$\\
\midrule
\endhead
\multirow{2}{*} {$0$} & \multirow{2}{*} {$1$} & $(11, 2)$ & $\{(5.436, 8.388)\}$ &  $\{[9.719,   20.567]\}$ \\
              & & $(6, 25)$ & $\{(3.239,   9.312)\}$ & $\{[5.432,   23.528]\}$ \\
\midrule
\multirow{2}{*} {$0.3$} & \multirow{2}{*} {$0.7$} & $(11, 2)$ & $\{(5.385, 8.413)\}$ &  $\{[8.973,   19.399]\}$ \\
              & & $(6, 25)$ & $\{(3.036,   9.393)\}$ & $\{[5.734,   20.654]\}$ \\
\midrule
\multirow{2}{*} {$0.6$} & \multirow{2}{*} {$0.4$} & $(11, 2)$ & $\{(5.392, 8.414)\}$ &  $\{[8.940,  19.344]\}$ \\
              & & $(6, 25)$ & $\{(3.057,   9.403)\}$ & $\{[5.763,   19.859]\}$ \\
\midrule
\multirow{2}{*} {$1$} & \multirow{2}{*} {$0$} & $(11, 2)$ & $\{(5.494, 8.375)\}$ &  $\{[10.279,  20.810]\}$ \\
              & & $(6, 25)$ & $\{(3.313,   9.305)\}$ & $\{[6.005,   21.640]\}$ \\
\bottomrule
\end{longtable}
\begin{figure}[h]
\begin{center}
\subfigure[For $w=0$ and starting point is $(6, 25)$]
    {\includegraphics[scale=0.27]{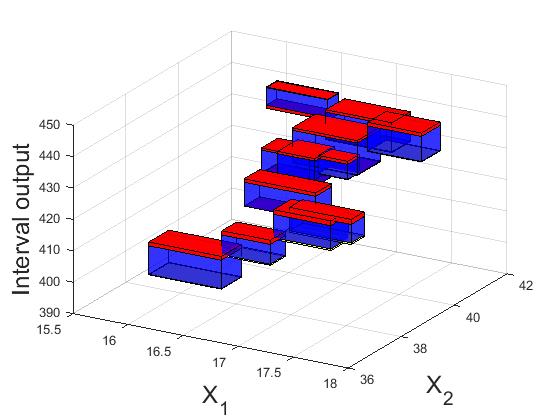}}
    \hfill
\subfigure[For $w=0$ and starting point is $(11, 2)$]
    {\includegraphics[scale=0.27]{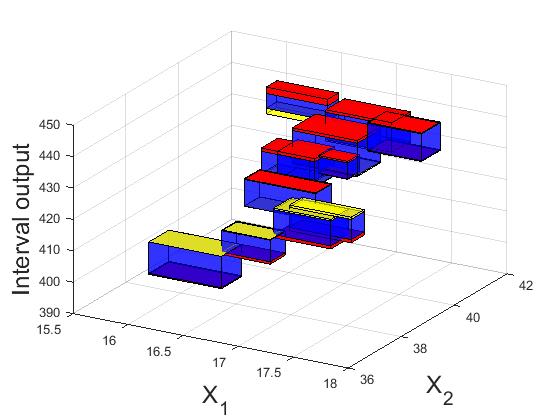}}
    \hfill
\subfigure[For $w=0.3$ and starting point is $(6, 25)$]
    {\includegraphics[scale=0.27]{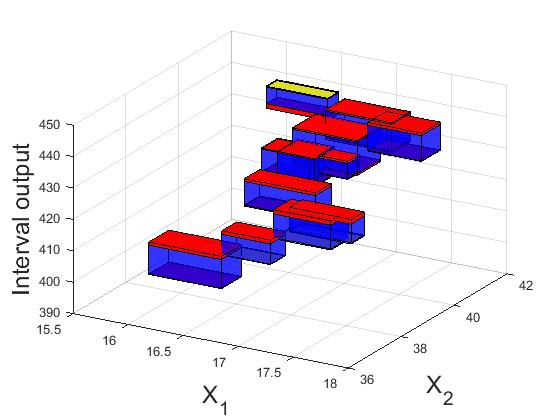}}
    \hfill
\subfigure[For $w=0.3$ and starting point is $(11, 2)$]
    {\includegraphics[scale=0.27]{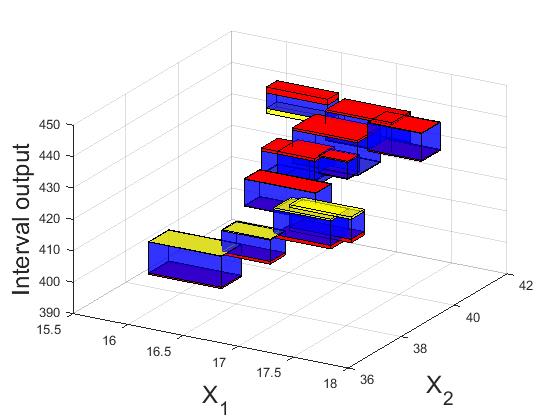}}
    \hfill
\subfigure[For $w=0.6$ and starting point is $(6, 25)$]
    {\includegraphics[scale=0.27]{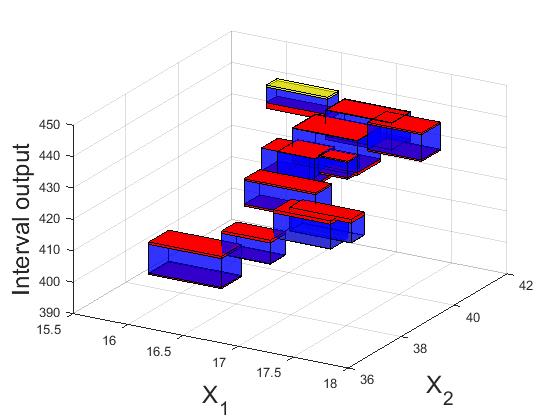}}
    \hfill
\subfigure[For $w=0.6$ and starting point is $(11, 2)$]
    {\includegraphics[scale=0.27]{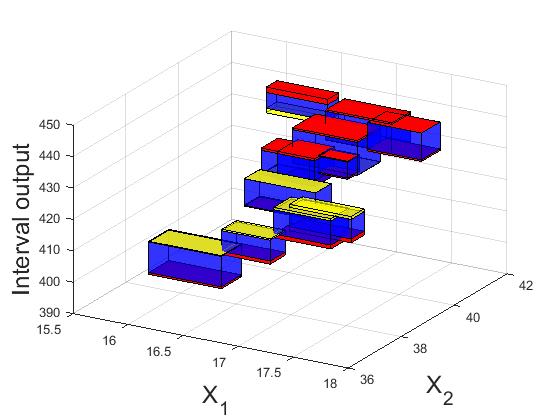}}
\subfigure[For $w=1$ and starting point is $(6, 25)$]
    {\includegraphics[scale=0.27]{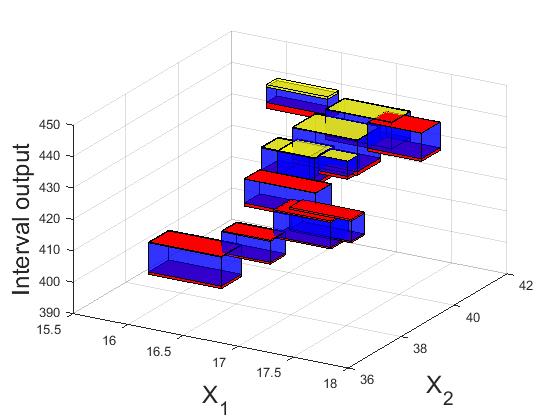}}
    \hfill
\subfigure[For $w=1$ and starting point is $(11, 2)$]
    {\includegraphics[scale=0.27]{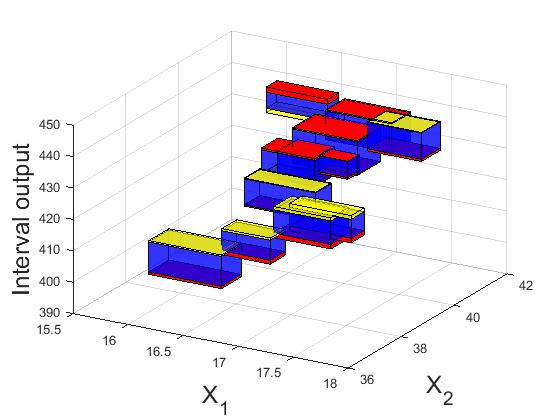}}
\caption{Comparison of actual interval-valued outputs and estimated interval-valued outputs of the given interval-valued features $\textbf{X}_k$, where the common portions of actual outputs and corresponding estimated outputs, the extended portions of actual outputs, and the extended portions of estimated outputs are illustrated by blue, red, and yellow regions, respectively.}\label{fsgmlasso}
\end{center}
\end{figure}

\section{Conclusion and Future Directions} \label{scf}

\noindent In this article, the $gH$-subgradient method to obtain efficient solutions to an unconstrained convex IOP and its algorithmic implementation (Algorithm \ref{algosg}) are presented. The convergence of the proposed method is studied (Theorem \ref{tcsm}). Considering an example (Example \ref{exsgm}), each step of the $gH$-subgradient method is explained. Using the derived technique a lasso problem with interval-valued features is solved. \\

In connection  with  the  present work, future research can be focus on the following works:

\begin{itemize}
    \item The boundedness of the step length $\alpha$ of the proposed $gH$-subgradient method.
    \item A $gH$-subgradient method to obtain efficient solutions to constrained convex IOPs.
    \item A $gH$-subgradient method to obtain efficient solutions to a dual IOP.
\end{itemize}


\section*{References}


\begin{thebibliography}{99}
	
\bibitem{Beck2003} Beck, A. and Teboulle, M. (2003). Mirror descent and nonlinear projected subgradient methods for convex optimization, \emph{Operations Research Letters}, 31(3), 167--175.

\bibitem{Bertsekas1999} Bertsekas, D. P. (1999). Nonlinear Programming, \emph{Athena Scientific}, second edition.

\bibitem{Bhurjee2012} Bhurjee, A. K. and Panda, G. (2012). Efficient solution of interval optimization problem, \emph{Mathematical Methods of Operations Research}, 76, 273--288.


\bibitem{Boyd2004} Boyd, S. and Vandenberghe, L. (2004), Convex optimization, \emph{Cambridge university press}.



\bibitem{Chalco2013kkt} Chalco-Cano, Y., Lodwick, W. A., and Rufian-Lizana, A. (2013). Optimality conditions of type KKT for optimization problem with interval-valued objective function via generalized derivative, \emph{Fuzzy Optimization and Decision Making}, 12, 305--322.

\bibitem{Chalco2013-2} Chalco-Cano, Y., Rufian-Lizana, A., Rom\'{a}n-Flores H., and Jim\'{e}nez-Gamero M. D. (2013). Calculus for interval-valued functions using generalized Hukuhara derivative and applications, \emph{Fuzzy Sets and Systems}, 219, 49--67.

\bibitem{Chanas1996} Chanas, S. and Kuchta, D. (1996). Multiobjective programming in optimization of interval objective functions--a generalized approach, \emph{European Journal of Operational Research}, 94(3), 594--598.


\bibitem{Chen2004} Chen, S. H., Wu, J., and Chen, Y. D.(2004). Interval optimization for uncertain structures, \emph{Finite Elements in Analysis and Design}, 40, 1379--1398.

\bibitem{Chen2020} Chen, S. L. (2020). The KKT optimality conditions for optimization problem with interval-valued objective function on Hadamard manifolds, \emph{Optimization}, 1--20.


\bibitem{Ghosh2016newton} Ghosh, D. (2017). Newton method to obtain efficient solutions of
    the optimization problems with interval-valued objective functions, \emph{Journal of Applied Mathematics and Computing}, 53, 709--731.

\bibitem{Ghosh2017spc} Ghosh, D.,  Ghosh, D., Bhuiya, S. K., and Patra, L. K. (2018). A
    saddle point characterization of efficient solutions for interval optimization
    problems, \emph{Journal of Applied Mathematics and Computing},  58(1--2), 193--217.

\bibitem{Ghosh2017quasinewton} Ghosh, D. (2017). A quasi-Newton method with rank-two update to solve interval optimization problems, \emph{International Journal of Applied and Computational Mathematics} 3(3), 1719--1738.

\bibitem{Ghosh2019extended} Ghosh, D., Singh, A.,  Shukla, K. K., and Manchanda, K. (2019). Extended Karush-Kuhn-Tucker condition for constrained interval optimization problems and its application in support vector machines, \emph{Information Sciences}, 504, 276--292.

\bibitem{Ghosh2019derivative} Ghosh, D.,  Chauhan, R. S., Mesiar, R., and Debnath, A. K. (2020).
    Generalized Hukuhara G\^{a}teaux and Fr\'{e}chet derivatives of
    interval-valued functions and their application in
    optimization with interval-valued functions, \emph{Information Sciences}, 510, 317--340.

\bibitem{Ghosh2020ordering} Ghosh, D., Debnath, A. K., and Pedrycz, W. (2020).
    A variable and a fixed ordering of intervals and their application in optimization with interval-valued functions, \emph{International Journal of Approximate Reasoning}, 121, 187--205.


\bibitem{Ghosh2019gradient} Ghosh, D.,  Debnath, A. K., Chauhan, R. S., and Mesiar, R. (2020). Generalized-Hukuhara-gradient efficient-direction method to solve optimization problems with interval-valued functions and its application in least squares problems, \emph{arXiv preprint arXiv:2011.10462}.

\bibitem{Hai2018} Hai, S. and Gong, Z. (2018). The differential and subdifferential for fuzzy mappings based on the generalized difference of n-cell fuzzy-numbers, \emph{Journal of Computational Analysis and Applications}, 24(1), 184--195.

\bibitem{Hukuhara1967} Hukuhara, M. (1967). Int\'{e}gration des applications measurables dont la valeur est un compact convexe, \emph{Funkcialaj Ekvacioj}, 10,
205--223.

\bibitem{Ishibuchi1990} Ishibuchi, H. and Tanaka, H. (1990). Multiobjective programming
    in optimization of the interval objective
        function, \emph{European Journal of Operational Research}, 48(2), 219--225.


\bibitem{Jiang2015} Jianga, C., Xiea, H. C., Zhanga, Z. G., and Hana, X. (2015). A new interval optimization method considering tolerance design, \emph{Engineering Optimization}, 47(12), 1637--1650.

\bibitem{Kiwiel1983} Kiwiel, K. C. (1983). An aggregate subgradient method for nonsmooth convex minimization, \emph{Mathematical Programming}, 27(3), 320--341.

\bibitem{Liu2007} Liu, S. T. and Wang, R. T. (2007). A numerical solution method to interval quadratic programming, \emph{Applied Mathematics and Computation}, 189(2), 1274--1281.


\bibitem{Lupulescu2014} Lupulescu, V. (2015). Fractional calculus for interval-valued functions, \emph{Fuzzy Sets and Systems}, 265, 63--85.

\bibitem{Markov1979} Markov, S. (1979). Calculus for interval functions of a real variable, \emph{Computing}, 22(4), 325--337.

\bibitem{Moore1966} Moore, R. E. (1966). \emph{Interval Analysis}, Prentice-Hall, Englewood Cliffs, New Jersey.

\bibitem{Moore1987} Moore, R. E. (1987). \emph{Method and applications of interval analysis}, Society for Industrial and Applied Mathematics.

\bibitem{Nedic2001} Nedic, A. and Bertsekas, D. P. (2001), Incremental subgradient methods for nondifferentiable optimization, \emph{SIAM Journal on Optimization}, 12(1), 109--138.

\bibitem{Nesterov2009} Nesterov, Y. (2009). Primal-dual subgradient methods for convex problems, \emph{Mathematical Programming}, 120(1), 221--259.	


\bibitem{Stefanini2008} Stefanini, L. (2008). A generalization of Hukuhara difference. In Soft Methods for Handling Variability and Imprecision, \emph{Springer, Berlin, Heidelberg}, 203--210.

\bibitem{Stefanini2009} Stefanini, L. and Bede, B. (2009). Generalized Hukuhara differentiability of interval-valued functions and
interval differential equations, \emph{Nonlinear Analysis}, 71, 1311--1328.

\bibitem{Stefanini2019} Stefanini, L. and Arana-Jim{\'e}nez, M. (2019). Karush-Kuhn-Tucker conditions for interval and fuzzy optimization in several variables under total and directional generalized differentiability, \emph{Fuzzy Sets and Systems}, 362, 1--34.

\bibitem{Studniarski1989} Studniarski, M. (1989). An algorithm for calculating one subgradient of a convex function of two variables, \emph{Numerische Mathematik}, 55(6), 685--693.


\bibitem{Wu2007} Wu, H. C. (2007). The Karush-Kuhn-Tucker optimality conditions in an optimization problem with interval-valued objective function, \emph{European Journal of Operational Research},  176, 46--59.

\bibitem{Wu2008} Wu, H. C. (2008). On interval-valued non-linear programming problems, \emph{Journal of Mathematical Analysis and Applications}, 338(1), 299--316.


\end{thebibliography}
\end{document}